\RequirePackage[l2tabu, orthodox]{nag}

\documentclass[12pt]{amsart}
\usepackage{fullpage,url,amssymb,enumerate}
\usepackage[all]{xy} 
\usepackage{comment}
\usepackage{graphicx}

\usepackage[OT2,T1]{fontenc}

\usepackage{amsthm}

\usepackage{color}

\def\bC{\mathbb{C}}

\def\cL{\mathcal{L}}
\def\cM{\mathcal{M}}

\def\cO{\mathcal{O}}
\def\bP{\mathbb{P}}
\def\cP{\mathcal{P}}
\def\bQ{\mathbb{Q}}

\def\cV{\mathcal{V}}

\def\bZ{\mathbb{Z}}

\def\barM{\overline{\cM}}

\def\wt{\widetilde}

\DeclareMathOperator{\Coeff}{Coeff}
\DeclareMathOperator{\dom}{dom}
\DeclareMathOperator{\ev}{ev}
\DeclareMathOperator{\exc}{exc}
\DeclareMathOperator{\Gr}{Gr}

\DeclareMathOperator{\id}{id}

\DeclareMathOperator{\Jac}{Jac}

\DeclareMathOperator{\Pic}{Pic}

\DeclareMathOperator{\spine}{sp}

\DeclareMathOperator{\vdim}{vdim}
\DeclareMathOperator{\vir}{vir}

\newtheorem{thm}{Theorem}
\newtheorem{lem}[thm]{Lemma}
\newtheorem{cor}[thm]{Corollary}
\newtheorem{prop}[thm]{Proposition}

\newtheorem{rem}[thm]{Remark}

\newtheorem{defn}[thm]{Definition}
\newtheorem{conjecture}[thm]{Conjecture}

\newcommand{\EE}{{\mathsf{E}}}
\newcommand{\Inc}{{\mathsf{Inc}}}
\newcommand{\PP}{{\mathsf{P}}}
\newcommand{\Tev}{{\mathsf{Tev}}}
\newcommand{\vTev}{{\mathsf{vTev}}}

\makeatletter
\g@addto@macro\bfseries{\boldmath} 
\makeatother

\usepackage{microtype}

\usepackage[
	backref,
	pdfauthor={}, 
	pdftitle={tevelev_hypersurface},
]{hyperref}

\begin{document}

\title{Asymptotic geometric Tevelev degrees of hypersurfaces}

\author{Carl Lian}

\address{Humboldt-Universit\"at zu Berlin, Institut f\"ur Mathematik,  Unter den Linden 6
\hfill \newline\texttt{}
 \indent 10099 Berlin, Germany} \email{{\tt liancarl@hu-berlin.de}}

\begin{abstract}
Let $(C,p_1,\ldots,p_n)$ be a fixed general pointed curve and let $(X,x_1,\ldots,x_n)$ be a smooth hypersurface of degree $e$ and dimension $r$ with $n$ {general} points. We consider the problem of enumerating maps $f:C\to X$ of degree $d$ (as measured in the ambient projective space) such that $f(p_i)=x_i$. When $e$ is small compared to $r$ and $d$ is large compared to $g$, $e$, and $r$, these numbers have been computed first by passing to a virtual count in Gromov-Witten theory obtained by Buch-Pandharipande, and then by showing (in work of the author with Pandharipande) that the virtual counts are enumerative via an analysis of boundary contributions in the moduli space of stable maps. In this note, we give a simpler computation via projective geometry.
\end{abstract}

\maketitle

\section{Introduction}

\subsection{Tevelev degrees}
Let $X$ be a smooth projective variety over $\mathbb{C}$ of dimension $r$, and let $\beta\in H_2(X,\bZ)$ be a non-zero effective curve class. Fix $g\geq 0$ and $n\geq 0$ in the stable range $2g-2+n>0$ so that the moduli space of stable curves $\barM_{g,n}$ is non-empty.

Let $\barM_{g,n}(X,\beta)$ the moduli space of stable maps, and let $\cM_{g,n}(X,\beta)$ be the open subspace of maps with smooth domain. We consider the canonical morphisms
\begin{align*}\label{tev_maps}
\tau&:\cM_{g,n}(X,\beta)\to\cM_{g,n}\times X^n\\
\overline{\tau}&:\barM_{g,n}(X,\beta)\to\barM_{g,n}\times X^n
\end{align*}
remembering in the first factor the (stabilized) domain curve of a stable map, and in the second factor the images of the marked points.

We will be concerned with the (virtual) degrees of the maps $\tau,\overline{\tau}$, counting maps $f:C\to X$ of degree $\beta$ with $f(p_i)=x_i$, where $(C,p_1,\ldots,p_n)$ is a fixed general curve and $x_1,\ldots,x_n\in X$ are general points. The moduli space of stable maps
  $\barM_{g,n}(X,\beta)$ has virtual dimension equal to the dimension of $\barM_{g,n}\times X^n$
  if and only if
  \begin{equation}\label{vir_dims_match}
    \int_\beta c_1(T_X)=r(n+g-1)\, .
  \end{equation}
If the dimension constraint \eqref{vir_dims_match} holds, then we expect to find a 
finite number of maps in question. We will assume \eqref{vir_dims_match}, often without mention, throughout this paper.
  
If it is further the case that all \textit{dominating} components of $\cM_{g,n}(X,\beta)$ are generically smooth of the expected dimension, then we may define the geometric Tevelev degrees of $X$ to be the corresponding degrees of $\tau$.

\begin{defn}\label{geom_tev_def}
Assume that \eqref{vir_dims_match} holds and that all components of $\cM_{g,n}(X,\beta)$ dominating $\cM_{g,n}\times X^n$ under the map
\begin{equation*}
\tau:\cM_{g,n}(X,\beta)\to\cM_{g,n}\times X^n
\end{equation*}
are generically smooth of the expected dimension. Then, we define the \textbf{geometric Tevelev degree} $\Tev^{X}_{g,\beta,n}\in\bZ$ to be the (set-theoretic) degree of $\tau$.
\end{defn}

By definition, the geometric Tevelev degree $\Tev^X_{g,\beta,n}$ counts maps with \textit{smooth} domain. This count, when defined, is a priori always a non-negative integer because automorphism groups never appear in a general fiber of $\tau$, see \cite[Remark 4]{lp}. 

As is well-known, the components of $\cM_{g,n}(X,\beta)$ can behave wildly and in particular have the wrong dimension, in which case the geometric Tevelev degrees may not be well-defined. On the other hand, upon passing to $\barM_{g,n}(X,\beta)$, one gains access to a virtual fundamental class living in the expected homological dimension, and its pushforward by $\tau$ allows one to define virtual Tevelev degrees in greater generality. 

\begin{defn}
Assume that \eqref{vir_dims_match} holds. The \textbf{virtual Tevelev degree} $\vTev^{X}_{g,\beta,n}\in\bQ$ is defined by
\begin{equation*}
\overline{\tau}_{*}([\barM_{g,n}(X,\beta)]^{\vir}) = \vTev^{X}_{g,\beta,n} \cdot [\barM_{g,n}\times X^n] \in A^0(\barM_{g,n}\times X^n),
\end{equation*}
where 
\begin{equation*}
\overline{\tau}:\barM_{g,n}(X,\beta)\to\barM_{g,n}\times X^n
\end{equation*}
is as above.
\end{defn}

While the enumeration of maps with fixed domain, viewed from various perspectives, is an old question, the systematic study of Tevelev degrees has received renewed interest after a result of Tevelev \cite[Theorem 6.2]{Tev} showing that $\Tev^{\bP^1}_{g,(g+1)[\bP^1],g+3}=2^g$ (see Definition \ref{geom_tev_def}) for all $g\ge0$. We review some of the subsequent results.

\subsection{Calculations of virtual Tevelev degrees}

The virtual Tevelev degrees are, at least in principle, easier to compute, owing to the following formula of Buch-Pandharipande.

\begin{thm}\cite[Theorem 1.3]{bp}\label{virtev_formula}
Assume that \eqref{vir_dims_match} holds. Then, we have:
\begin{equation*}
\vTev^X_{g,\beta,n}=\Coeff(\PP^{\star n}\star \EE^{\star g},q^\beta \PP).
\end{equation*}
\end{thm}
Here, $\PP\in H^{*}(X)\subset QH^{*}(X)$ denotes the point class in the small quantum cohomology ring of $X$, $\EE$ denotes the \textit{quantum Euler class}
\begin{equation*}
\EE=\sum_j \gamma_j^{\vee}\star\gamma_j \in QH^{*}(X)
\end{equation*}
where $\{\gamma_j\},\{\gamma^{\vee}_j\}$ are dual bases for $H^{*}(X)$ with respect to the (classical) cup product, $\star$ denotes the quantum product, and $\Coeff(\alpha,q^{\beta}\PP)$ denotes the $q^{\beta}\PP$-coefficient when $\alpha$ is expanded as a power series in the quantum variable $q$.

From Theorem \ref{virtev_formula}, one can obtain explicit formulas for virtual Tevelev degrees in a number of examples. In the case $X=\bP^r$, we have:
\begin{thm}\cite[(3)]{bp}\label{proj_vir}
Assume that \eqref{vir_dims_match} holds, that is, that $n=\frac{r+1}{r}\cdot d-g+1$. Then,
\begin{equation*}
\vTev^{\bP^r}_{g,d,n}=(r+1)^g.
\end{equation*}
\end{thm}
Here, we use $\beta=d$ to denote the class of a curve of degree $d$. Note that in order for $n=\frac{r+1}{r}\cdot d-g+1$ to be an integer, we require $r|d$. A closely related virtual count was obtained by Bertram-Daskalopoulos-Wentworth before the systematic development of virtual fundamental classes in Gromov-Witten theory, see \cite[Theorem 2.9]{BDW}.

For low degree hypersurfaces, there is a similar result. For a hypersurface $X_e\subset \mathbb{P}^{r+1}$ of degree $e$, condition \eqref{vir_dims_match} becomes
\begin{equation}\label{vir_dims_match_hypersurface}
n=\frac{r+2-e}{r}\cdot d-g+1.
\end{equation}

\begin{thm}\cite[Theorem 1.5]{bp}\label{virtualtev_hypersurface} Let
    $X_e\subset \mathbb{P}^{r+1}$ be a smooth hypersurface of degree
    $e$ and dimension $r$. Suppose that $3\le e\le\frac{r+3}{2}$. Then, for all $g,d,n\ge0$ satisfying $2g-2+n>0$ and \eqref{vir_dims_match_hypersurface}, we have
    \begin{equation*}
    \vTev^{X_e}_{g,d,n} =  ((e-1)!)^{n} \cdot (r+2-e)^g\cdot  e^{(d-n)e-g+1}.
    \end{equation*}
   
   \end{thm}
The case of quadric hypersurfaces $(e=2)$ is exceptional, see \cite[Theorem 1.4]{bp}.

A non-trivial aspect of the proof of Theorem \ref{virtualtev_hypersurface} is controlling the contributions from primitive cohomology of $X$. The explicit computation of $\vTev^{X_e}_{g,d,n}$ in the remainder of the Fano range $\frac{r+4}{2}\le e\le r+1$ remains open, but Cela \cite{cela} has given partial results.

See also \cite[\S3-4]{bp} for results on $\vTev^{X}_{g,\beta,n}$ when $X=G/P$ (which also includes the case of quadric hypersurfaces).

\subsection{Calculations of geometric Tevelev degrees}

The computation of geometric Tevelev degrees is more subtle. Complete results were first given in the case $X=\bP^1$ by Cela-Pandharipande-Schmitt \cite{cps} via intersection theory on Hurwitz spaces of branched covers. Brill-Noether theory guarantees that the geometric Tevelev degrees of $X=\bP^1$ (in fact, of any $\bP^r$) are well-defined for all $d\ge1$.
\begin{thm}\cite[Theorem 6]{cps}\label{cps_thm}
For all $g,d,n\ge0$ satisfying $2g-2+n>0$ and $n=2d-g+1$, we have:
\begin{equation*}
\Tev^{\bP^1}_{g,d,n}=2^g-\sum_{i=0}^{g-d-1}\binom{g}{i}+(g-d-1)\binom{g}{g-d}+(d-g-1)\binom{g}{g-d+1}.
\end{equation*}
\end{thm}
Binomial coefficients $\binom{g}{i}$ with $i<0$ are interpreted to vanish. In particular, $\Tev^{\bP^1}_{g,d,n}=2^g$ whenever $d\ge g+1$ (equivalently, $d\le n-2$), agreeing with the virtual count of Theorem \ref{proj_vir}.

Later work of the author with Farkas gave the following alternate formula via limit linear series and Schubert calculus.
\begin{thm}\cite[Theorem 1.3]{fl}\label{fl_thm}
For all $g,d,n\ge0$ satisfying $2g-2+n>0$ and $n=2d-g+1$, we have:
\begin{equation*}
\Tev^{\bP^1}_{g,d,n}=\int_{\Gr(2,d+1)}\sigma_1^g\sum_{i+j=2d-2-g}\sigma_i\sigma_j.
\end{equation*}
\end{thm}
That the expression on the right evaluates to $2^g$ when $d\ge g+1$ is non-trivial; a brute force calculation using the Pieri rule is possible, but a combinatorial proof via the RSK algorithm was given by Gillespie-Reimer-Berg \cite{grb}. The formulas of Theorems \ref{cps_thm} and \ref{fl_thm} have been extended to counts of covers of $\bP^1$ with arbitrary ramification profiles, see \cite{cl}.

The calculation of $\Tev^{\bP^r}_{g,d,n}$ for $r>1$ is considerably more difficult and will be addressed in forthcoming work \cite{l_coll}. However, when $d$ is large compared to $r$ and $g$, it is again true that the geometric and virtual Tevelev degrees agree:
\begin{thm}\cite[Theorem 1.2]{fl}
Suppose that $d\ge rg+r$, or equivalently that $d\le n-2$. Then, for all $g\ge0$, we have:
\begin{equation*}
\Tev^{\bP^r}_{g,d,n}=(r+1)^g.
\end{equation*}
\end{thm}

This leads to the following conjecture.

\begin{conjecture}\label{enum_conj}
Let $X$ be a smooth, projective Fano variety and let $g$ be a fixed genus. Then, there exists a constant $C_{X,g}$ depending only on $X$ and $g$ such that, if $\int_{\beta}c_1(T_X)>C_{X,g}$, then
\begin{equation*}
\vTev^X_{g,\beta,n}=\Tev^X_{g,\beta,n}
\end{equation*}
That is, the virtual Tevelev degree is enumerative when $\beta$ is sufficiently positive. (In particular, the geometric Tevelev degree is well-defined.)
\end{conjecture}

Partial results in this direction have been given by the author with Pandharipande:

\begin{thm}\label{lian_pand}
Conjecture \ref{enum_conj} holds for (among other examples):
\begin{itemize}
\item\cite[Theorem 10]{lp} homogeneous varieties $X=G/P$,
\item\cite[Theorem 11]{lp} hypersurfaces $X_e\subset\bP^{r+1}$ of degree $e$, where $r>(e+1)(e-2)$.
\end{itemize}
\end{thm}

Combining the enumerativity for hypersurfaces in Theorem \ref{lian_pand} with the virtual calculations of Theorem \ref{virtualtev_hypersurface} yields an explicit \textit{asymptotic} calculation of geometric Tevelev degrees of hypersurfaces of low degree, that is, when the degree $d$ of $f:C\to X$ is sufficiently large. We state the conclusion for completeness.

\begin{thm}\label{main_thm}
    Let $X_e\subset \mathbb{P}^{r+1}$ be a smooth hypersurface of degree
    $e$ and dimension $r$. Suppose that $e\ge3$ and $r>(e+1)(e-2)$. Then
    \begin{equation*}
    \Tev^{X_e}_{g,d,n} =  ((e-1)!)^{n} \cdot (r+2-e)^g\cdot  e^{(d-n)e-g+1}
    \end{equation*}
    \emph{for all $d$ sufficiently large (compared to $g,e,r$).}
%
\end{thm}

\subsection{Geometric Tevelev degrees of hypersurfaces via projective geometry}

Given the simplicity of the formulas of Theorem \ref{main_thm}, the authors of \cite{bp} speculate that a proof via projective geometry along the lines of the asymptotic calculation of \cite{lp} for $X=\bP^r$ should be possible (see also \cite[Question 28]{lp}). In this paper, we give such a proof.

We first outline the strategy; assume first for simplicity that $g=0$, where the key idea already appears. Let $V=H^0(\bP^1,\cO(d))$ and let $\bP(V^{r+2})\cong\bP^{(r+2)(d+1)-1}$ denote the space of $(r+2)$-tuples $[f_0:\cdots:f_{r+1}]$, where $f_i\in H^0(\bP^1,\cO(d))$, the $f_j$ taken up to simultaneous scaling. A generic choice of $f_j$ defines a map $f:\bP^1\to\bP^{r+1}$ of degree $d$, but the degree of this map drops when the $f_j$ share common factors.

Now, fix general points $p_1,\ldots,p_n\in\bP^1$ and $x_1,\ldots,x_n\in X_e=Z(F)\subset \bP^{r+1}$. The condition that either $f(p_i)=x_i$ or all of $f_0,\ldots,f_{r+1}$ vanish at $p$ defines a linear subspace $\Inc(p_i,x_i)\subset \bP(V^{r+2})$ of codimension $r+1$. We wish to intersect the $\Inc(p_i,x_i)$ with the locus $J'_F$ consisting of $f$ with image in $X_e$, that is, with $F(f_0,\ldots,f_{r+1})=0$. The subscheme $J'_F$ is an intersection of $de+1$ hypersurfaces of degree $e$ in $\bP(V^{r+2})$.

One quickly sees that the loci $\Inc(p_i,x_i)$ and $J'_F$ do not intersect in the expected dimension, owing to the fact that the $x_i$ are general points in $X_e$, not in $\bP^{r+1}$. This is easily circumvented by imposing instead the condition that $f(p_i)\in L_i$, where $L_i\subset\bP^{r+1}$ is a fixed general line through $x_i$; the corresponding subscheme $\Inc(p_i,L_i)\subset\bP(V^{r+2})$ now has the correct expected codimension of $r$ upon restriction to $J'_F$, but introduces an extra factor of $e$ in the intersection number, coming from the $e-1$ other points of $L_i\cap X$.

There is, however, a more serious problem: typically, the intersection of the $\Inc(p_i,L_i)$ and $J'_F$ contains positive-dimensional excess loci corresponding to degenerate $f$. For example, if $d\ge n$, we may take $f_0,\ldots,f_{r+1}$ to be constant multiples of each other vanishing on all of $p_1,\ldots,p_n$, such that the image of the constant map $f=[f_0:\cdots:f_{r+1}]$ lies in $X_e$. As \eqref{vir_dims_match} implies that
\begin{equation*}
n\sim\frac{r+2-e}{r}\cdot d,
\end{equation*}
we have that if $e\ge 3$, even when $d$ becomes sufficiently large, such degenerate maps $f$ always exist.

We resolve the issue by blowing up the linear subspaces on $\bP(V^{r+2})$ where all of $f_0,\ldots,f_{r+1}$ vanish at one of the $p_i$. More explicitly, let $$Y\subset \bP(V^{r+2})\times(\bP^{r+1})^{n}$$
be the incidence correspondence (cut out by bilinear equations) consisting of $(f,x'_1,\ldots,x'_n)$ for which either $f_0,\ldots,f_{r+1}$ simultaneously vanish on $p_i$ or $f(p_i)=x'_i$. Then, we cut out (in a precise way) a subscheme $J_F\subset Y$ approximating the \emph{proper transform} $J'_F\subset\bP(V^{r+2})$.

The key observation is that, on the locus where $f$ has a simple base-point $p_i$, if $\wt{f}$ denotes the map of degree $d-1$ obtained by dividing by the common factor, then $(f,x'_1,\ldots,x'_n)\in J_F$ if and only if the points $\wt{x}_i:=\wt{f}(p_i)$ and $x'_i$ lie on a line contained in $X_e$ (Proposition \ref{T_settheoretic}).

This new condition may alternatively be regarded as the requirement that there exist a stable map of degree $d$ to $X_e$ defined by the degree $d-1$ map $\wt{f}:\bP^1\to X_e$ and a rational tail mapping isomorphically to the line on $X_e$ through $\wt{x}_i$ and $x'_i$. In light of this comparison, we employ a similar analysis as in \cite{lp} to argue that, under the conditions of Theorem \ref{main_thm}, the intersection of $J_F$ and the pullback to $Y$ of $$\prod_{i=1}^{n}L_i\subset(\bP^{r+1})^n$$is a reduced, zero-dimensional subscheme supported on the locus of $f$ which are well-defined at all of the $p_i$. In particular, the degree of this zero-dimensional subscheme is equal to $e^n\cdot\Tev^{X_e}_{0,d,n}$.

When $g$ is arbitrary, we relativize this construction over the Jacobian $\Jac^d(C)$ (following the strategy of \cite{BDW}) to prove Theorem \ref{main_thm}. 

We will record, along the way of the proof, a new effective bound on $d$, above which the virtual Tevelev degrees are enumerative. We now state the more refined result.

\begin{thm}\label{main_thm_refined}
    Let $X_e\subset \mathbb{P}^{r+1}$ be a smooth hypersurface of degree
    $e$ and dimension $r$. Suppose that $e\ge3$ and $r>(e+1)(e-2)$.
    
    Then,
    \begin{equation*}
    \Tev^{X_e}_{g,d,n} =  ((e-1)!)^{n} \cdot (r+2-e)^g\cdot  e^{(d-n)e-g+1}
    \end{equation*}
    for \emph{all} $d$ if $g=0$, and whenever
    \begin{equation*}
d>\frac{r((3g-2)(1+e)+1+g(r+2))}{r-(e+1)(e-2)}
\end{equation*}
if $g>0$.
\end{thm}

%

%

%

That no condition on $d$ is necessary when $g=0$ is also \cite[Corollary 34]{lp}.

In fact, for a \textit{general} hypersurface $X_e$ of degree $e\le r-1$, then the irreducibility of the full Kontsevich space $\overline{M}_{0,n}(X_e,d)$ (due to Riedl-Yang \cite[Theorem 3.3]{ry}, and also Harris-Roth-Starr \cite{hrs}, Beheshti-Kumar \cite{bk} in smaller ranges for $e$) implies that the virtual Tevelev degree $\vTev^{X_e}_{0,d,n}$ is enumerative. In particular, if $3\le e\le \frac{e+3}{2}$ and $X_e$ is general, then the formula $\Tev^{X_e}_{0,d,n} =  ((e-1)!)^{n} \cdot  e^{(d-n)e+1}$ holds for a fixed pair of integers $(d,n)$ satisfying \eqref{vir_dims_match_hypersurface}, see Remark \ref{general_genus0}. (To guarantee the formula to hold on $X_e$ for \textit{any} $(d,n)$, one should take $X_e$ to be \textit{very} general.) We thank Eric Riedl for pointing this out.

We also note here that, while we will not need the arguments of \cite{lp} in full generality, the transversality analysis of our new construction confronts the same issues of controlling excess dimensions of curves on hypersurfaces, and in particular, produces the same bound $r>(e+1)(e-2)$ for (asymptotic) enumerativity. Thus, our proof of Theorem \ref{main_thm_refined}, which was previously established through the virtual calculation of \cite{bp} and then the enumerativity analysis of \cite{lp}, may be viewed in comparison in the following light. First, we give a more elementary path to the counts obtained as virtual degrees in \cite{bp}. Then, we apply special cases of the arguments of \cite{lp} to show that these counts are equal to the desired geometric Tevelev degrees.

\subsection{Further directions}\label{further}

We make several remarks concerning extensions of this work.

\begin{enumerate}
\item We expect the method to go through for complete intersections of low degree, see also \cite[\S 5]{bp} and \cite[Proposition 26]{lp}.
\item We will see that our calculation uses the fact that the degree $e$ of the hypersurface $X_e$ is at least 3 in an essential way, see step (3) of \S\ref{setup_section}. However, we believe that minor modifications should allow one to handle the exceptional case $e=2$ similarly.
\item Our transversality analysis in \S\ref{transversality_section}, which makes no improvements on the method of \cite{lp}, seems to go through highly suboptimal bounds. We expect more refined estimates of excess dimensions of families of stable maps on hypersurfaces to allow for stronger results.
\item While our calculation (provably) produces a geometrically enumerative count only under the conditions given in Theorem \ref{main_thm_refined}, our construction produces a cycle class of homological dimension 0 and degree equal to the corresponding virtual Tevelev degree under the much milder bounds $r\ge 2e-3$ and $d\ge2g$. Thus, it seems plausible that excess contributions in our geometric setup can be identified with boundary contributions to the virtual Tevelev degrees in greater generality.
\item When $e>\frac{r+3}{2}$, our method fails because because the construction imposes no non-trivial condition on $f$ whenever $f$ has a \emph{double} base-point at $x_i$, which allows $f$ to vanish to order 2 at \textit{all} $x_i$ whenever $d\ge 2n$. However, working on further (iterated) blowups of $\bP(V^{r+2})$ should allow one to extend the calculation to larger degrees $e$ and predict formulas for the corresponding virtual Tevelev degrees, see \cite{cela}.
\item\label{insertion_intro} One can consider more general insertions at the marked points of $X$, which will be carried out in a later version of \cite{bp}, see also \cite{buch_zoom}. More precisely, for $i=1,2,\ldots,n$, fix a cohomology class $\Omega_i\in H^{m_i}(X)$. Then, define ${\langle \Omega_1,\ldots,\Omega_n\rangle}^{X}_{g,n}$ by the integral
\begin{equation*}
\int_{\barM_{g,n}\times X^n}\overline{\tau}_{*}([\barM_{g,n}(X,\beta)]^{\vir})\cdot([(C,p_1,\ldots,p_n)]\otimes\Omega_1\otimes\cdots\otimes\Omega_n)\in\bQ,
\end{equation*}
where $[(C,p_1,\ldots,p_n)]\in A_0(\barM_{g,n})$ is the class of any (general) curve and 
\begin{equation*}
\overline{\tau}:\barM_{g,n}(X,\beta)\to\barM_{g,n}\times X^n
\end{equation*}
is the forgetful morphism as before. The case $\Omega_i=[x_i]\in H^{2r}(X)$ corresponds to the virtual Tevelev degree $\vTev^X_{g,n,\beta}$ defined earlier. Note that ${\langle \Omega_1,\ldots,\Omega_n\rangle}^{X}_{g,n}$ vanishes unless
\begin{equation*}
\int_\beta c_1(T_X)=r(g-1)+\frac{1}{2}\sum_{i=1}^{n}m_i.
\end{equation*}

When $X=X_e\subset\bP^{r+1}$ is a smooth hypersurface and the $\Omega_i$ are pullbacks of powers of the hyperplane class on $\bP^{r+1}$, our geometric calculation produces more generally a zero-dimensional cycle of degree equal (up to a factor of $e^n$, as in the case $\Omega_i=[x_i]$) to ${\langle \Omega_1,\ldots,\Omega_n\rangle}^{X}_{g,n}$, see \S\ref{insertions}. We expect again these counts to be enumerative for fixed $X,g$ and $d$ large, but we do not consider this aspect of the question here.
\end{enumerate}

\subsection{Acknowledgments}

We thank Anders Buch, Alessio Cela, Gavril Farkas, Maria Gillespie, Rahul Pandharipande, Eric Riedl, Johannes Schmitt, and Jason Starr for discussions about Tevelev degrees in various contexts. We also thank the referee for a thorough reading and for pointing out a gap in the main construction in an earlier version, among numerous other helpful comments. This project was completed with the support of an NSF postdoctoral fellowship, grant DMS-2001976, and the MATH+ Incubator grant ``Tevelev Degrees.'' 
 
\section{Proof of Theorem \ref{main_thm_refined}}

\subsection{Preliminaries}\label{prelim_section}

Let $X=Z(F)\subset\bP^{r+1}$, where $F\in H^0(\bP^{r+1},\cO(e))$, be a smooth hypersurface of degree $e\ge3$, and let $x_1,\ldots,x_n\in X$ be general points. (We will drop the subscript $e$ from the introduction denoting the degree of $X$ in the rest of the paper.) Fix a general pointed curve $(C,p_1,\ldots,p_n)$ and general lines $L_1,\ldots,L_n\subset\bP^{r+1}$ such that $L_i\ni x_i$ for each $i$. In particular, we assume that the $L_i$ are pairwise disjoint and all intersect $X$ transversely.

We assume throughout that $d\ge2g$. It is straightforward to check that this is implied by the lower bound on $d$ imposed by Theorem \ref{main_thm_refined}. Let $\Jac^d(C)$ be the Jacobian of degree $d$ line bundles on $C$, let $\pi:C\times\Jac^d(C)\to\Jac^d(C)$ be the projection, and let $\cP$ be the Poincar\'{e} line bundle on $C\times\Jac^d(C)$. Now, define the vector bundle $\cV=\pi_{*}(\cP)$ of rank $d-g+1$ on $\Jac^d(C)$. 

Let $\nu:\bP(\cV^{\oplus r+2})\to \Jac^d(C)$ be the projective bundle whose points parametrize $(r+2)$-tuples $f=[f_0:\cdots:f_{r+1}]$ of sections of $H^0(C,\cL)$ up to simultaneous scaling. (We will generally drop the reference to $\cL$ in referring to a point of $\bP(\cV^{\oplus r+2})$.) Given such an $f=[f_0:\cdots:f_{r+1}]$, we say that a point $p\in C$ is a \textit{base-point} for $f$ if $f_0,\ldots,f_{r+1}$ all vanish at $p$, a \textit{double base-point} if they do so to order \textit{at least} 2, and a \textit{simple base-point} if $p$ is a base-point that is not a double base-point.

We denote by $\widetilde{f}:C\to\bP^{r+1}$ the morphism of degree at most $d$ obtained from $f$ by twisting down all base-points of $f$, and (abusively) by $f:C\dashrightarrow\bP^{r+1}$ to be the \textit{rational} map given simply by $f=[f_0:\cdots:f_{r+1}]$, which is undefined at any base-points of $f$.


\subsection{Blowing up the locus of maps with base-points}\label{blowup_section}
We work on the ambient space $\bP(\cV^{\oplus r+2})\times\bP^{r+1}$. Fix any point $p\in C$, and let $Y_p\subset \bP(\cV^{\oplus r+2})\times\bP^{r+1}$ be the locus of $(f,y)$ where either $p$ is a base-point of $f$ or $f(p)=y$. Alternatively, we have a linear projection map $\ev:\bP(\cV^{\oplus r+2})\dashrightarrow\bP^{r+1}$ sending $f$ to $f(p)$, defined exactly where $p$ is not a base-point of $f$. Then, under
\begin{equation*}
(\ev,\id): \bP(\cV^{\oplus r+2})\times \bP^{r+1}\dashrightarrow \bP^{r+1}\times\bP^{r+1},
\end{equation*}
we have that $Y_p$ is the closure of the pre-image of the diagonal $\Delta\subset \bP^{r+1}\times\bP^{r+1}$.

Upon composing with the projection to $\bP(\cV^{\oplus r+2})$, we may identify $b:Y_p\to \bP(\cV^{\oplus r+2})$ with the blowup of $\bP(\cV^{\oplus r+2})$ at the locus where $f_0,\ldots,f_{r+1}$ all vanish at $p$. Because we have assumed that $d\ge 2g$, this locus is the projective sub-bundle $\bP(\cV(-p)^{\oplus r+2})\subset\bP(\cV^{\oplus r+2})$ of $(r+2)$-tuples of sections of $H^0(C,\cL(-p))$ over each point $[\cL]$ of $\Pic^d(C)$.

The exceptional divisor $E\subset Y_p$ is equal to $\bP(\cV(-p)^{\oplus r+2})\times\bP^{r+1}$. Let $H,H_p\in\Pic(Y_p)$ denote the pullbacks of the relative hyperplane class on $\bP(\cV^{\oplus r+2})$ and the hyperplane class of $\bP^{r+1}$ to $Y_p$, respectively.

\begin{lem}\label{exc_class}
We have $[E]=H-H_p$ in $\Pic(Y_p)$.
\end{lem}

\begin{proof}
As a Cartier divisor on $Y_p$, we may identify $E$ with the locus of $([f_0:\cdots:f_{r+1}],[y_0:\cdots:y_{r+1}])\in Y_p$ where $f_0(p)=0$, which has class $H$, minus the locus where $y_0=0$, which has class $H_p$.
\end{proof}

\begin{defn}
For $k=0,1,\ldots,e$, we define $I^{(k)}_{F,p}\subset \bP(\cV^{\oplus r+2})$ to be the locus of $[f_0:\cdots:f_{r+1}]$ for which $F(f_0,\ldots,f_{r+1})\in H^0(C,\cL^{\otimes e})$ vanishes to order at least $k$ at $p$.
\end{defn}

We now construct inductively a chain of subschemes
$$J^{(e)}_{F,p}\subset\cdots \subset J^{(0)}_{F,p}=Y_p$$
where $J^{(k)}_{F,p}\subset b^{-1}(I^{(k)}_{F,p})$ may be regarded as the "virtual proper transform" of $I^{(k)}_{F,p}$ under $b$. Assuming that $J^{(k-1)}_{F,p}\subset b^{-1}(I^{(k-1)}_{F,p})$, the subscheme $J^{(k)}_{F,p}\subset J^{(k-1)}_{F,p}$ is cut out by the vanishing of a section of a certain line bundle.

We first cut out a subscheme $\overline{J}^{(k)}_{F,p}\subset J^{(k-1)}_{F,p}$ by the divisorial condition that $F$ (which, by assumption, vanishes to order at least $k-1$ at $p$) vanishes to order $k$ at $p$. This amounts to the vanishing of a section
$$F^{(k)}\in \ker(H^0(C,\cL^{\otimes e}|_{kp})\to H^0(C,\cL^{\otimes e}|_{(k-1)p}))\cong\frac{H^0(C,\cL^{\otimes e}(-(k-1)p))}{H^0(C,\cL^{\otimes e}(-kp))}$$
obtained by restricting $F(f_0,\ldots,f_{r+1})\in H^0(C,\cL^{\otimes e})$. (Here, by $kp$ we mean the subscheme of $C$ of multiplicity $k$ supported at $p$, and by $\cL^{\otimes e}|_{kp}$ we mean $\cL^{\otimes e}\otimes_{\cO_C}\cO_{kp}$.)

Globally, $\overline{J}^{(k)}_{F,p}$ is the vanishing locus of a canonical section
$$s^{(k)}_F\in H^0\left(J^{(k-1)}_{F,p},b^{*}\left(\nu^{*}\pi_{*}\left(\frac{\cP^{\otimes e}(-(k-1)p)}{\cP^{\otimes e}(-kp)}\right)\otimes\cO_{\bP}(e)\right)\right).$$
Note in particular that $\deg(\cL^{\otimes e}(-kp))=ed-k\ge e(d-1)>2g-2$ by assumption, so $\cP^{\otimes e}(-(k-1)p)$ and $\cP^{\otimes e}(-kp)$ have vanishing higher cohomology on the fibers of $\pi:C\times\Jac^d(C)\to\Jac^d(C)$.
 
We now express this vanishing locus in terms of local coordinates. Let $t$ be an analytic local coordinate on $C$ near $p$, so that $\cL|_{\widehat{\cO_{C,p}}}\cong\bC[[t]]$ simultaneously for all $\cL\in\Jac^d(C)$. More precisely, $t\in \widehat{\cO_{C,p}}$ is obtained by restriction from an everywhere non-vanishing section of the trivial line bundle $\cP|_{\Jac^d(C)\times \widehat{\cO_{C,p}}}$.

Then, write
\begin{equation*}
f_j=f_{j,0}+f_{j,1}t+f_{j,2}t^2+\cdots
\end{equation*}
for the Taylor expansion of $f_j$ around $p$, so that
\begin{equation*}
F(f_0,\ldots,f_{r+1})=s^{(1)}_F+s^{(2)}_Ft^{1}+\cdots+s^{(k-1)}_Ft^{k-2}+s^{(k)}_Ft^{k-1}+\cdots.
\end{equation*}

By assumption, the equations $s^{(1)}_F,s^{(2)}_F,\ldots,s^{(k-1)}_F$ (which are all degree $e$ polynomials in the $f_{j,\ell}$) all vanish on $J^{(k-1)}_{F,p}$, and $\overline{J}^{(k)}_{F,p}=V(s^{(k)}_F)$. By construction, we have $\overline{J}^{(k)}_{F,p}\subset b^{-1}(I^{(k)}_{F,p})$.

Note also that in these local coordinates, the exceptional divisor $E$ is the common vanishing locus of the equations $f_{j,0}$, and $s_F^{(k)}$ has total degree at least $e+1-k$ in the $f_{j,0}$. Alternatively, on the principal open subset $D_{+}(y_\ell)$ of $\bP^{r+1}$, we may substitute $f_{j,0}=f_{\ell,0}\cdot\frac{y_{j}}{y_\ell}$ for $j=0,1,\ldots,r+1$, and $f_{\ell,0}$ is a local equation cutting out the exceptional divisor $E$. Then, the equation $s^{(k)}_F$ is divisible by $f_{\ell,0}^{e+1-k}$. 

We therefore conclude that the section $s^{(k)}_F$ vanishes \textit{globally} upon restriction to the Cartier divisor $(e+1-k)E$, and thus is the image of a canonical section
$$\wt{s}^{(k)}_F\in H^0\left(J^{(k-1)}_{F,p},b^{*}\left(\nu^{*}\pi_{*}\left(\frac{\cP^{\otimes e}(-(k-1)p)}{\cP^{\otimes e}(-kp)}\right)\otimes\cO_{\bP(\cV^{\oplus r+2})}(e)\right)\otimes\cO_{Y_p}(-(e+1-k)E)\right).$$
Finally:
\begin{defn}\label{divided_def}
Define $J^{(k)}_{F,p}\subset J^{(k-1)}_{F,p}$ by $V(\wt{s}^{(k)}_F)$.
\end{defn}

Clearly, $J^{(k)}_{F,p}\subset\overline{J}^{(k)}_{F,p}\subset b^{-1}(I^{(k)}_{F,p})$. Locally, the section $\wt{s}^{(k)}_F$ is given by $s^{(k)}_F/f_{\ell,0}^{e+1-k}$ on $D_{+}(y_\ell)$ as above. We do not make the claim that $J^{(k)}_{F,p}$ does not have a component supported on $E$, so $J^{(k)}_{F,p}$ may be strictly larger than the proper transform of $I^{(k)}_{F,p}$. Furthermore, we do not claim that $J^{(k)}_{F,p}$ has the expected codimension of $k$ in $Y_p$. 

However, as a result of the construction, we have a canonical excess class $[J^{(e)}_{F,p}]^{\exc}\in A_{\dim(Y_p)-e}(J^{(e)}_{F,p})$ with the property that, upon restriction of the construction to any open neighborhood of $Y_p$ in which $J^{(e)}_{F,p}$ is pure of the expected codimension of $e$, we recover the fundamental class. Namely,

\begin{lem}\label{exc_class_virproptrans}
We have
\begin{equation*}
[J^{(e)}_{F,p}]^{\exc}=\left(\prod_{k=1}^{e}((k-1)H+(e+1-k)H_p)\right)\cap[Y_p]
\end{equation*}
in $A_{\dim(Y_p)-e}(J^{(e)}_{F,p})$.
\end{lem}

\begin{proof}
$J^{(e)}_{F,p}\subset Y_p$ is cut out successively by the vanishing of sections of the line bundles $$b^{*}\left(\nu^{*}\pi_{*}\left(\frac{\cP^{\otimes e}(-(k-1)p)}{\cP^{\otimes e}(-kp)}\right)\otimes\cO_{\bP(\cV^{\oplus r+2})}(e)\right)\otimes\cO_{Y_p}(-(e+1-k)E)$$ for $k=1,2,\ldots,e$. The claim now follows from Lemma \ref{exc_class} and the triviality of the line bundle $\pi_{*}\left(\frac{\cP^{\otimes e}(-(k-1)p)}{\cP^{\otimes e}(-kp)}\right)$.
\end{proof}

\subsection{Main setup}\label{setup_section}
We now set up the intersection theory calculation computing $\Tev^{X}_{g,d,n}$. In the first two steps, we repeat the construction of \S\ref{blowup_section} for the marked points $p_1,\ldots,p_n$ taken together. Then, we impose the condition that our map $f:C\to\bP^{r+1}$ lie on $X$ globally (as opposed to ``to order $e$" at the $p_i$). Finally, we impose the incidence conditions at the marked points.

\begin{enumerate}
\item[(1)] Let $Y\subset \bP(\cV^{\oplus r+2})\times(\bP^{r+1})^{n}$ be the intersection of the loci $Y_{p_i}$ pulled back from each of the $n$ factors $(\bP^{r+1})^{n}$, as defined in the previous section. Explicitly, $Y$ is the locus of $(f,x'_1,\ldots,x'_n)$ where, for each $i$, either $p_i$ is a base-point of $f$ or $f(p_i)=x'_i$.


\item[(2)] We define a closed subscheme $J^{(e)}_F\subset Y$ by iterating the construction of \S\ref{blowup_section} for $p_1,\ldots,p_n$, one point at a time. More precisely, in the first step, we first define $J^{(e)}_{F,p_1}\subset Y$ by pulling back the construction of $J^{(e)}_{F,p_1}\subset Y_{p_1}$ to $Y$. In the second, we carry out the construction of $J^{(e)}_{F,p_2}$ now on the ambient space given by the pullback of $J^{(e)}_{F,p_1}$ to $Y$ instead of on $Y_{p_2}$, cutting out the successive conditions that $F$ vanish to order $k=1,2,\ldots,e$ at $p_2$, now having already imposed the additional condition of $e$-order vanishing at $p_1$. There are $n$ steps in total.

At the $m$-th step of this construction, we use the fact that the line bundle $\cL^{\otimes e}(-ep_1-\cdots-ep_{m-1}-kp_{m})$ has degree at least
\begin{equation}\label{cbcineq}
e(d-n)=e\left(d\cdot\frac{e-2}{r}+g-1\right)>2g-2\tag{$*$}
\end{equation}
when $k\le e$, when imposing the condition that a canonical section of $\cL^{\otimes e}(-ep_1-\cdots-ep_{m-1}-kp_{m}))$ vanishes to order $k$ at $p_m$.

We obtain in the end a subscheme $J^{(e)}_{F}\subset Y$ of expected (but not necessarily actual) codimension $ne$, equal to the intersection of the pullbacks of the $J^{(e)}_{F,p_j}\subset Y_{p_j}$ to $Y$, and contained in the locus of $(f,x'_1,\ldots,x'_n)$ for which $F(f_0,\ldots,f_{r+1})$ vanishes to order $e$ at all of the $p_i$. We also get an excess cycle $[J^{(e)}_{F}]^{\exc}\in A_{\dim(Y)-ne}(J^{(e)}_{F})$ in a similar way to before, from the first Chern classes of the line bundles whose sections cut out $J^{(e)}_{F}$.

\item[(3)] Now, on $J^{(e)}_F$, because the section $F(f_0,\ldots,f_{r+1})\in H^0(C,\cL^{\otimes e})$ vanishes to order $e$ at $p_1,\ldots,p_n$, it may be regarded naturally as a section of $H^0(C,\cL^{\otimes e}(-ep_1-\cdots-ep_n))$. We then define $J_F\subset J^{(e)}_F$ by the vanishing of this section.

Globally, we have a canonical section 
\begin{equation*}
s'_F\in H^0(\bP(\cV^{\oplus r+2}),\nu^{*}(\pi_{*}(\cP^{\otimes e}(-ep_1-\cdots-ep_n))\otimes\cO_{\bP}(e))),
\end{equation*}
and $J_F\subset J^{(e)}_F$ is by definition equal to $V(s'_F)$. We have applied here the inequality \eqref{cbcineq} to ensure that the line bundle $\cP^{\otimes e}(-ep_1-\cdots-ep_n)$ has vanishing higher cohomology on the fibers of $\pi$ (note that it is crucial here that $e\ge3$). Note in particular that $F(f_0,\ldots,f_{r+1})$ vanishes (now to every order at every point) everywhere on $J_F$.

As in step (2), we do not assume that $V(s'_F)$ has the expected codimension of $e(d-n)-g+1$. However, we get an excess class
$$[J_{F}]^{\exc}:=c_{e(d-n)-g+1}(\nu^{*}(\pi_{*}(\cP^{\otimes e}(-ep_1-\cdots-ep_n))\otimes\cO_{\bP}(e)))\cap[J^{(e)}_{F}]^{\exc}$$
in $A_{\dim(Y)-(de-g+1)}(J_{F})$ agreeing with the usual fundamental class upon restriction to any open neighborhood of $Y$ in which we have dimensional transversality for both steps (2) and (3).

\item[(4)] Finally, we define $T\subset J_F$ by the condition that $x'_i\in L_i$ for all $i$; $T$ is cut out by $r$ linear conditions on each of the $n$ factors $\bP^{r+1}$.
\end{enumerate}

The expected codimension of $T\subset \bP(\cV^{\oplus r+2})\times(\bP^{r+1})^n$ is
\begin{equation*}
n(r+1)+en+[e(d-n)-g+1]+rn=g+[(r+2)(d-g+1)-1]+(r+1)n,
\end{equation*}
so the expected dimension of $T$ is 0. When $(f,x'_1,\ldots,x'_n)\in T$ and $f$ has no base-points, the degree $d$ morphism $f:C\to X$ has $f(p_i)\in L_i\subset X$.

\subsection{Transversality}\label{transversality_section}

We now analyze in more detail the construction above, to show that, under the conditions of Theorem \ref{main_thm_refined}, $T$ is in fact a reduced zero-dimensional scheme whose points correspond (up to a factor of $e^{n}$, owing to the extra intersections of $L_i$ with $X$) to maps enumerated in $\Tev^X_{g,d,n}$. There are three steps:

\begin{enumerate}
\item Describe the points of $T$ set-theoretically (Proposition \ref{T_settheoretic})
\item Show, under the conditions of Theorem \ref{main_thm_refined}, that all points of $T$ correspond to maps counted in $\Tev^X_{g,d,n}$ (Corollary \ref{T_boundary_cor})
\item Show that $T$ is reduced of dimension 0 at all such points (Proposition \ref{transverse})
\end{enumerate}

\begin{prop}\label{T_settheoretic}
Suppose that $(f,x'_1,\ldots,x'_n)\in T$. Let $\wt{f}:C\to\bP^{r+1}$ be the map defined by $f$ as in \S\ref{prelim_section}, obtained by twisting down base-points. Then, $\wt{f}$ has image in $X$, and in addition satisfies the property that for all $i=1,2,\ldots,n$, one of the following holds:
\begin{enumerate}
\item $p_i$ is not a base-point of $f$ and $\wt{f}(p_i)=x'_i$, or
\item $p_i$ is a simple base-point of $f$, and $\wt{x}_i:=\wt{f}(p_i)$ and $x'_i$ lie on a line contained in $X$ (or are equal), or
\item $p_i$ is a double base-point of $f$.
\end{enumerate}
\end{prop}

\begin{proof}
That $f$ has image in $X$ is guaranteed by steps (2) and (3) of the construction of $T$, as the section $F(f_0,\ldots,f_{r+1})\in H^0(C,\cL^{\otimes e})$ is constrained to vanish. If $p_i$ is not a base-point of $f$, then step (1) guarantees that $\wt{f}(p_i)=x'_i$. Now, it remains to consider the case in which $f$ has a simple base-point of $f$ at $p_i$.

We adopt the local coordinates from \S\ref{blowup_section} above, that is, the construction of $J^{(e)}_{F,p}\subset Y_{p}\subset\bP(\cV^{\oplus r+2})\times\bP^{r+1}$, where we take $p=p_i$. If $f$ has a simple base-point at $p$, that is, the $f_{j,0}$ vanish for $j=0,\ldots,r+1$, but the $f_{j,1}$ are not all 0, then we have
\begin{equation*}
\wt{x}:=\wt{x}_i=[f_{0,1}:\cdots:f_{r+1,1}].
\end{equation*}

Consider now the equations $\wt{s}^{(k)}_F$. As in the discussion preceding Definition \ref{divided_def}, we restrict to the principal open subset $D_{+}(y_\ell)$ of $\bP^{r+1}$, where we may replace the equations $f_{j,0}$ with $f_{\ell,0}\cdot\frac{y_{j}}{y_\ell}$, then divide the equation $s^{(k)}_F$ by $f_{\ell,0}^{e+1-k}$ to obtain $\wt{s}^{(k)}_F$. For convenience, we also multiply $s^{(k)}_F$ by the unit $y_\ell^{e+1-k}$. Then, upon restriction to the locus $f_{\ell,0}=0$ (that is, the exceptional divisor, where $f$ has a base-point at $p$), the only surviving terms of the equations $\wt{s}^{(k)}_F$, for $k=0,1,\ldots,e-1$, are the monomials given by $e-k$ factors of the form $y_{j}$ and $k$ factors of the form $f_{j,1}$. 

Now, for indeterminates $\alpha,\beta$, consider the expansion
\begin{equation*}
F(\alpha f_{0,1}+\beta y_{0},\ldots,\alpha f_{r+1,1}+\beta y_{r+1})
\end{equation*}
as a polynomial of total degree $e$ in the $f_{j,1},y_{j}$, where we evaluate the function $F$ at a point $\alpha\cdot \wt{x}+\beta\cdot y\in\bP^{r+1}$ lying on the line between $\wt{x}$ and $x'=y=[y_{0}:\cdots:y_{r+1}]$. Then, the sum of the terms of total bidegree $(k,e-k)$ in the $f_{j,1},y_{j}$, respectively, is precisely $\alpha^k\beta^{e-k}\wt{s}^{(k+1)}_F$, which vanishes for $k=0,1,\ldots,e-1$ by the construction of $J^{(e)}_F$, and for $k=e$ by the fact that $F_f=0$ on $J_F$.

In particular, any point on the line between $\wt{x}$ and $x'$ vanishes on $F$, so we obtain the needed conclusion.
\end{proof}

We now analyze the ``degenerate'' points of $T$ parametrizing $f$ having base-points. We will require the following lemma, cf. \cite[Proof of Proposition 26]{lp}.

\begin{lem}\label{h1_bound}
Let $D$ be a smooth curve of genus $g$, and let $j:D\to X$ be a morphism of degree $d\ge0$. Then
\begin{equation*}
h^1(D,j^{*}T_X)\le de+1+g(r+2).
\end{equation*}
\end{lem}

\begin{proof}
From the short exact sequence 
\begin{equation*}
0\to T_X|_D\to T_{\bP^{r+1}}|_{D}\to \cO_{\bP^{r+1}}(e)|_{D} \to 0,
\end{equation*}
we have
\begin{equation*}
h^1(D,j^{*}T_X)\le h^0(D,\cO_{\bP^{r+1}}(e))+h^1(D,T_{\bP^{r+1}}).
\end{equation*}
To bound the first term, we have 
\begin{equation*}
h^0(D,\cO_{\bP^{r+1}}(e)) = de-g+1+h^1(D,\cO_{\bP^{r+1}}(e)) \le de+1.
\end{equation*}
For the second, the Euler sequence
\begin{equation*}
0\to \cO_{\bP^{r+1}}\to  \cO_{\bP^{r+1}}(1)^{\oplus r+2} \to T_{\bP^{r+1}} \to 0
\end{equation*}
shows that $h^1(D,T_{\bP^{r+1}})\le h^1(D,\cO_{\bP^{r+1}}(1)^{\oplus r+2})\le g(r+2)$. Combining yields the claim.
\end{proof}

Suppose now that, for general choices of $(C,p_1,\ldots,p_n)\in\cM_{g,n}$ and $x_i\in L_i\subset X$, there exists a point $(f,x'_1,\ldots,x'_n)\in T$ of the form that $f$ has:
\begin{itemize}
\item $b_1$ simple base-points at $p_1,\ldots,p_{b_1}$,
\item $b_2$ double base-points at $p_{n-b_2+1},\ldots,p_{n}$, and
\item $b_0$ base-points (\textit{counted with multiplicity}) outside of the $p_i$.
\end{itemize}
Assume that $b_0,b_1,b_2$ are not all 0, that is, $f$ fails to be base-point-free. 

By Lemma \ref{T_settheoretic}, $\wt{x}_1,\ldots,\wt{x}_{b_1}$ may be joined by lines on $X$ with points $x''_1,\ldots,x''_{b_1}$, respectively, where $x''_i\in L_i\cap X$. We use this fact to construct a stable map from a certain nodal curve to $X$. Namely, let $D$ be the union of a genus $g$ spine $(D_{\spine},p_1,\ldots,p_n)\cong (C,p_1,\ldots,p_n)$ containing the points $p_i$, with smooth rational tails $R_1,\ldots,R_{b_1}$ attached to $D_{\spine}$ at $p_1,\ldots,p_{b_1}$, respectively. Let $p'_i\in R_i$ be an arbitrary smooth point, for $i=1,2,\ldots,b_1$. 

Then, we construct the stable map $[\overline{f}:D\to X]\in \barM_{g,n-b_2}(X,d')$ by mapping $D_{\spine}$ to $X$ via $\wt{f}$, and by mapping each $R_i$ isomorphically to the line between $\wt{f}(p_i)$ and $x''_i$, such that $\overline{f}(p'_i)=x''_i$. Denote the degree of $\overline{f}$ by $d'=d-b_0-2b_2$.

The map $\overline{f}$ lives in a stratum
\begin{equation*}
\cM_{\Gamma}\subset \barM_{g,n-b_2}(X,d')
\end{equation*}
parametrizing stable maps whose domain curve is homeomorphic to $D$. (When $b_1=0$, we have that $\cM_{\Gamma}$ is the open stratum of stable maps with smooth domain.)

\begin{prop}\label{T_boundary}
Let $\tau:\cM_{\Gamma}\to\barM_{g,n-b_2}\times X^{n-b_2}$ be the forgetful map. Assume that at least one of $b_0,b_1,b_2$ is strictly positive. Then, under the assumptions of Theorem \ref{main_thm_refined}, the map $\tau$ is not dominant.
\end{prop}

\begin{proof}

Let $\delta=\vdim(\barM_{g,n}(X,d))=\dim(\barM_{g,n}\times X^{n})$. Then, we have
\begin{equation*}
\dim(\barM_{g,n-b_2}\times X^{n-b_2})=\Delta-(r+1)b_2
\end{equation*}
and 
\begin{align*}
\vdim(\cM_{\Gamma})&=\vdim(\barM_{g,n-b_2}(X,d'))-b_1\\
&=\delta-(d-d')(r+2-e)-b_2-b_1\\
&\le \delta-(2r+5-2e)b_2-(r+2-e)b_0-b_1.
\end{align*}

Let $\cM_{\Gamma}^{\dom}$ be the union of the components of $\cM_{\Gamma}$ dominating the target of $\tau$. We now argue as in \cite[Proposition 22]{lp} to show that
\begin{equation*}
\dim(\cM_{\Gamma}^{\dom})<\dim(\barM_{g,n-b_2}\times X^{n-b_2}),
\end{equation*}
which will immediately yield a contradiction.

If $n-b_1-b_2\ge \max(2g,1)$, then \cite[Proposition 13, Lemma 16]{lp} implies that  $\dim(\cM_{\Gamma}^{\dom})=\vdim(\cM_{\Gamma})$, so we have
\begin{align*}
\vdim(\cM_{\Gamma})&=\delta-(2r+5-2e)b_2-(r+2-e)b_0-b_1\\
&<\delta-(r+1)b_2\\
&=\dim(\barM_{g,n-b_2}\times X^{n-b_2}),
\end{align*}
as long as $r\ge 2e-3$ and at least one of $b_0,b_1,b_2$ is strictly positive.

Otherwise, suppose that $n-b_1-b_2\le \max(2g-1,0)$. We claim first that
\begin{equation*}
\dim(\cM_{\Gamma}^{\dom})\le\vdim(\cM_{\Gamma})+h^1(D_{\spine},\overline{f}^{*}T_X).
\end{equation*}
Indeed, for a general stable map $\overline{f}$ in $\cM_{\Gamma}^{\dom}$, all components $R_i$ of $D$ other than $D_{\spine}$ are lines passing through a general point of $X$, onto which the restriction of $T_X$ is therefore globally generated. Indeed, we have $T_X|_{D}\cong\sum_{i=1}^{r}\cO(a_i)$ for some integers $a_i$, and the $a_i$ must be non-negative in order for the points $x_i=\overline{f}(p_i)$ to deform to first order in all tangent directions.

In particular, these lines move in a family of expected dimension (\cite[Proposition 13]{lp}), and the condition of passing through any fixed node locally has the expected codimension of $r$. As a result, the excess dimension $\dim(\cM_{\Gamma}^{\dom})-\vdim(\cM_{\Gamma})$ is bounded above by that of the stable map when restricted to $D_{\spine}$, which is $h^1(D_{\spine},\overline{f}^{*}T_X)$.

We now have
\begin{align*}
\dim(\cM_{\Gamma}^{\dom})&\le\vdim(\cM_{\Gamma})+h^1(D_{\spine},\overline{f}^{*}T_X)\\
&\le \vdim(\cM_{\Gamma})+(d'-b_1)e+1+g(r+2)\\
&\le \vdim(\cM_{\Gamma})+(d-n-b_2+ \max(2g-1,0))e+1+g(r+2)\\
&=\delta-(2r+5-2e)b_2-(r+2-e)b_0-b_1\\
&\qquad\qquad+(d-n-b_2+\max(2g-1,0))e+1+g(r+2)\\
&\le \delta-(2r+5-2e)b_2-(n-b_2-\max(2g-1,0)))\\
&\qquad\qquad+(d-n-b_2+\max(2g-1,0))e+1+g(r+2)\\
&=\delta-(2r+4-e)b_2-n+(d-n)e\\
&\qquad\qquad+\max(2g-1,0)(1+e)+1+g(r+2)\\
&\le \dim(\barM_{g,n-b_2}\times X^{n-b_2})-n+(d-n)e\\
&\qquad\qquad+\max(2g-1,0)(1+e)+1+g(r+2)\\
&=\dim(\barM_{g,n-b_2}\times X^{n-b_2})+d\cdot\frac{(e+1)(e-2)-r}{r}\\
&\qquad\qquad+[\max(2g-1,0)+(g-1)](1+e)+1+g(r+2)\\
&<\dim(\barM_{g,n-b_2}\times X^{n-b_2}),
\end{align*}
as long as $(e+1)(e-2)<r$ and 
\begin{equation*}
d>\frac{r((\max(2g-1,0)+(g-1))(1+e)+1+g(r+2))}{r-(e+1)(e-2)}.
\end{equation*}
When $(e+1)(e-2)<r$ and $g=0$, the lower bound on $d$ is negative, so we obtain the conclusion for all $d$; if $g>0$, the bound becomes
\begin{equation*}
d>\frac{r((3g-2)(1+e)+1+g(r+2))}{r-(e+1)(e-2)}.
\end{equation*}
\end{proof}

\begin{cor}\label{T_boundary_cor}
Any $(f,x'_1,\ldots,x'_n)\in T$ has $f$ base-point free.
\end{cor}

\begin{proof}
Proposition \ref{T_boundary} shows that $\cM_{\Gamma}$ does not dominate the target of the forgetful map
\begin{equation*}
\tau:  \barM_{g,n-b_2}(X,d')\to  \barM_{g,n-b_2}\times X^{n-b_2}.
\end{equation*}
Thus, varying over all possible $\Gamma$, we conclude that for a fixed general pointed curve $(C,p_1,\ldots,p_n)$, there exists a non-empty open set $U\subset X^n$ with the property that, for any $(x''_1,\ldots,x''_n)\in U$, there does not exist a stable map $\overline{f}$ as above with $(b_0,b_1,b_2)\neq(0,0,0)$ and $\overline{f}(p'_i)=x''_i$.

It now suffices to show that, for fixed $(x_1,\ldots,x_n)\in U$, a general choice of lines $L_i\ni x_i$ (assumed transverse to $X$) has the property that, for any choices of $x''_i\in L_i\cap X$, we have $(x''_1,\ldots,x''_n)\in U$. Suppose that this is not the case. Then, consider the (proper) map
$$\lambda:X^n\backslash U\to \prod_i\bP_i,$$
where $\bP_i\cong\bP^r$ is the space of lines through $x_i$, sending $(x''_1,\ldots,x''_n)\in X^n$ to the lines $L_i$ through $x_i,x''_i$. For dimensional reasons, $\lambda$ cannot be dominant, and a general collection of $(L_1,\ldots,L_n)$ in the complement of the image suffices.
\end{proof}

\begin{prop}\label{transverse}
Under the assumptions of Theorem \ref{main_thm_refined}, the subscheme $T$ is reduced of the expected dimension 0. In particular, the intermediate subschemes 
\begin{equation*}
J_F\subset J^{(e)}_F\subset Y\subset \bP(\cV^{\oplus r+2})\times(\bP^{r+1})^{n}
\end{equation*}
have the expected dimension in a neighborhood of $T$.
\end{prop}

\begin{proof}
The second claim will follow from the first. Corollary \ref{T_boundary_cor} shows that every point of $T$ corresponds to an honest map $f:C\to X$ of degree $d$ with $f(p_i)\in L_i$. Consider now a non-zero tangent vector $v$ of $T\subset \bP(\cV^{\oplus r+2})\times(\bP^{r+1})^{n}$. We first observe that, because the $L_i,X$ intersect transversely at the $x_i$, and because $f(p_i)$ has been constrained to lie on both $L_i,X$, we have that $v$ is trivial upon projection to the factor $(\bP^{r+1})^{n}$.

From here, one easily verifies that, in fact, $v$ defines a non-trivial relative tangent vector at $[f]$ to the forgetful morphism
\begin{equation*}
\tau:\cM_{g,n}(X,d)\to \cM_{g,n}\times X^n.
\end{equation*}
On the other hand, \cite[Proposition 14]{lp} shows that, over some open subset $V\subset X^n$ (for fixed $(C,p_1,\ldots,p_n)$), such a relative tangent vector cannot exist when $n\ge \max(2g,1)$. This inequality is immediate when $g=0$ and, as one easily checks, follows from the required lower bound on $d$ given in Theorem \ref{main_thm_refined} when $g>0$. The same argument from the proof of Corollary \ref{T_boundary_cor} guarantees that, for a general choice of lines $L_i$, we have $(x''_1,\ldots,x''_n)\in V$ for any $x''_i\in L_i\cap X$. The conclusion follows.
\end{proof}

\begin{cor}\label{en_factor}
We have $\deg(T)=e^n\cdot\Tev^X_{g,d,n}$.
\end{cor}

\begin{proof}
Corollary \ref{T_boundary_cor} and Proposition \ref{transverse} show on the one hand that, for a general choice of lines $L_i$, the fibers of the map
\begin{equation*}
\tau:\cM_{g,n}(X,d)\to \cM_{g,n}\times X^n.
\end{equation*}
over any of the $e^n$ points $((C,p_1,\ldots,p_n),x''_1,\ldots,x''_n)$ with $x''_i\in L_i\cap X$ are reduced of dimension 0 and degree $\Tev^X_{g,d,n}$, and on the other hand that, the union of these fibers is (scheme-theoretically) isomorphic to $T$ for the same choice of the $L_i$. The conclusion follows.
\end{proof}

\begin{rem}\label{general_genus0}
Let $X\subset\bP^{r+1}$ be a \textit{general} hypersurface of degree $e\le r-1$. Then, the Kontsevich space $\overline{M}_{0,n}(X,d')$ is irreducible and generically smooth of the expected dimension for all positive $d'\le d$ \cite{hrs,bk,ry}.

If we assume further that $3\le e\le \frac{r+3}{2}$, then the arguments of Proposition \ref{T_boundary} and Corollary \ref{T_boundary_cor} show that every point of $T$ comes from a base-point-free $f$. Here, we use only that $f$ cannot have a double base-point at all of the $p_i$, and do not require any estimates on $h^1(D_{\spine},\overline{f}^{*}T_X)$. Moreover, the arguments of Proposition \ref{transverse} and Corollary \ref{en_factor} go through as well. Thus, the degree of $T$, as computed in the next section, gives the genus 0 geometric Tevelev degrees of any general hypersurface $X$ (up to a factor of $e^n$).
\end{rem}

\subsection{Calculation}

We now compute the degree of $T$. We do so by following the construction; in each step, we obtain a new (excess) cycle class by capping with a class pulled back from the ambient space $\bP(\cV^{\oplus r+2})\times(\bP^{r+1})^{n}$. Let $H=c_1(\cO_{\bP}(1))$ be the relative hyperplane class on $\bP(\cV^{\oplus r+2})$, and let $H_1,\ldots,H_n$ be the hyperplane classes on the factors $\bP^{r+1}$.

\begin{enumerate}
\item[(1)] $Y\subset\bP(\cV^{\oplus r+2})\times(\bP^{r+1})^{n}$ is cut out by pulling back the class of the diagonal $\bP^{r+1}\subset\bP^{r+1}\times\bP^{r+1}$ under the rational map $\ev:\bP(\cV^{\oplus r+2})\times(\bP^{r+1})^{n}\dashrightarrow(\bP^{r+1}\times\bP^{r+1})^n$. Because $\ev$ is given by relative linear projections and the codimension of the base locus is strictly larger than that of the class being pulled back, the class of the subscheme $Y$ is
\begin{equation*}
[Y]=\left(\prod_{i=1}^{n}(H^{r+1}+H^{r}H_i+\cdots+H_i^{r+1})\right)\cap[\bP(\cV^{\oplus r+2})\times(\bP^{r+1})^{n}].
\end{equation*}

\item[(2)] Arguing as in Lemma \ref{exc_class_virproptrans}, we have
\begin{equation*}
[J^{(e)}_{F}]^{\exc}=\left(\prod_{i=1}^{n}\prod_{k=1}^{e}((k-1)H+(e+1-k)H_i)\right)\cap[Y]
\end{equation*}
in $A_{\dim(Y)-ne}(J^{(e)}_{F})$. Indeed, $J^{(e)}_{F}\subset Y$ is cut out by sections of the (pulled back) line bundles $$\cO_{\bP(\cV^{\oplus r+2})}(e)\otimes\cO_{Y}(-(e+1-k)E_i)$$ for $k=1,2,\ldots,e$ and exceptional divisors $E_i\subset Y_{p_i}$.

(Compare \cite[\S 5.3]{bp}.)

\item[(3)] Consider the top Chern class
\begin{equation*}
c_{t}(\nu^{*}(\pi_{*}(\cP^{\otimes e}(-ep_1-\cdots-ep_n))\otimes\cO_{\bP(\cV^{\oplus r+2})}(e)))
\end{equation*}
where $t=(d-n)e-g+1$ is the rank of the bundle $\pi_{*}(\cP^{\otimes e}(-ep_1-\cdots-ep_n))$ on $\Jac^d(C)$. This is equal to
\begin{equation*}
\sum_{m=0}^{g}c_m(\pi_{*}(\cP^{\otimes e}(-ep_1-\cdots-ep_n)))\cdot(eH)^{t-m}
\end{equation*}
where we have dropped the $\nu^{*}$ on the bundle pulled back from $\Jac^d(C)$. Note that the Chern classes in the first factor vanish above degree $\dim(\Jac^d(C))=g$.

Now, a standard computation (see e.g. \cite[Chapter VIII, \S2]{acgh}) shows that
\begin{equation*}
c_m(\pi_{*}(\cP^{\otimes e}(-ep_1-\cdots-ep_n)))=\frac{(-e^2)^m}{m!}\cdot\theta^m\in H^{2m}(\Jac^d(C)),
\end{equation*}
where $\theta\in H^{2}(\Jac^d(C))$ is the class of the theta divisor. 

Therefore,
\begin{equation*}
[J_F]^{\exc}=\left(e^{t}\cdot\sum_{m=0}^{g}\frac{(-e)^m}{m!}\cdot\theta^mH^{t-m}\right)\cap [J^{(e)}_F]^{\exc}.
\end{equation*}

\item[(4)] $T\subset J_F$ is cut out by $r$ linear conditions on each factor of $\bP^{r+1}$, and has the expected dimension of 0. Therefore,
\begin{equation*}
[T]=\left(\prod_{i=1}^{n}H_i^{r}\right)\cap[J_F]^{\exc}
\end{equation*}
Note that this holds even if the intermediate schemes $J_F,J^{(e)}_F$ have the wrong dimension; restricting the construction to a neighborhood of $T$ avoiding all excess components yields simply the fundamental class of $T$ in the end.
\end{enumerate}

\begin{prop}\label{degree}
The degree of $T$ as a 0-cycle on $\bP(\cV^{\oplus r+2})\times(\bP^{r+1})^{n}$ is
\begin{equation*}
(e!)^n\cdot(r+2-e)^g\cdot e^{(d-n)e+1}.
\end{equation*}
\end{prop}

\begin{proof}
Applying the projection formula, we need to multiply the formulas for the four classes computed above, and extract all coefficients with degree exactly $r+1$ in each $H_i$. Because we automatically get $r$ factors of $H_i$ from step (4) and one factor of $H_i$ from the term $(eH_i)$ in step (2), we must take the term $H^{r+1}$ in all $n$ factors in step (1), and the terms $(e-1)H,(e-2)H,\ldots,H$ each $n$ times (in addition to the contributions from $eH_i$) in step (2).

We therefore find that
\begin{equation*}
[T]=(H_1\cdots H_n)^{r+1}\cdot (e)!^n \cdot e^{t}\cdot H^{n(r+1)+(e-1)n}\cdot\sum_{m=0}^{g}\frac{(-e)^m}{m!}\cdot\theta^mH^{t-m},
\end{equation*}

Pushing forward to $\Jac^d(C)$ and applying the projection formula, we obtain
\begin{equation*}
\nu_{*}[T]=(e)!^n\cdot e^{t} \cdot\sum_{m=0}^{g}\frac{(-e)^m}{m!}\cdot\theta^m\cdot s_{g-m}(\cV^{\oplus r+2}).
\end{equation*}
Finally, \cite[Chapter VIII, \S2]{acgh} (see also \cite[Proof of Theorem 2.9]{BDW}) shows that 
\begin{equation*}
s_{g-m}(\cV^{\oplus r+2})=(r+2)^{g-m}\cdot\frac{\theta^{g-m}}{(g-m)!},
\end{equation*}
so we conclude that
\begin{align*}
\deg(T)&=(e)!^n\cdot e^{t} \cdot\frac{\deg(\theta^g)}{g!}\cdot\sum_{m=0}^{g}\binom{g}{m}(-e)^{m}(r+2)^{g-m}\\
&=(e!)^n\cdot(r+2-e)^g\cdot e^{(d-n)e-g+1}.
\end{align*}
where we apply Poincar\'{e}'s formula \cite[p. 25]{acgh} and have substituted back $t=(d-n)e-g+1$.
\end{proof}

We now conclude the main Theorem.

\begin{proof}[Proof of Theorem \ref{main_thm_refined}]
Combine Proposition \ref{degree} and Corollary \ref{en_factor}.
\end{proof}

\subsection{More general insertions}\label{insertions}

We observe here that one can just as easily replace the incidence condition in step (4) that $x'_i\in L_i$ with the condition that $x'_i$ lie in a linear space $X_i\cong\bP^{\ell_i}\subset\bP^{r+1}$. In order for the expected number of maps $f:C\to X$ of degree $d$ satisfying the property that $f(p_i)\in X_i$ to be finite, we require
\begin{equation*}
d(r+2-e)=r(n+g-1)-\sum_{i=1}^{n}\ell_i.
\end{equation*}
The corresponding virtual count is ${\langle \Omega_1,\ldots,\Omega_n\rangle}^{X}_{g,n}$, where $\Omega_i=[X_i\cap X]$, see (\ref{insertion_intro}) of \S\ref{further}.

One may construct a scheme $T_{\ell_1,\ldots,\ell_n}$ of \textit{expected} dimension 0 as above, parametrizing maps $f:C\to X$ of degree $d$ with $f(p_i)\in X_i$. Write now
\begin{equation*}
(1+z+\cdots+z^{r+1})\cdot\prod_{j=0}^{e-1}(j+(e-j)z)=\alpha_1z+\cdots+\alpha_{e+r+1}z^{e+r+1},
\end{equation*}
where $\alpha_1,\ldots,\alpha_{e+r+1}\in\bZ_{>0}$. If it happens that $T_{\ell_1,\ldots,\ell_n}$ has actual dimension 0, then proceeding as in Proposition \ref{degree}, we compute
\begin{equation*}
\deg(T_{\ell_1,\ldots,\ell_n})=(r+2-e)^g\cdot e^{(d-n)e-g+1}\cdot \prod_{i=1}^{n}\alpha_{\ell_i},
\end{equation*}
which matches the virtual degree ${\langle \Omega_1,\ldots,\Omega_n\rangle}^{X}_{g,n}$ upon dividing by $e^n$ as in the proof of Theorem \ref{main_thm}, see \cite[pp. 21-23]{buch_zoom}. We do not consider the question of transversality here.

\end{document}